\documentclass{article}

\usepackage[utf8]{inputenc}
\usepackage{lmodern}
\usepackage{textgreek}

\usepackage{amsmath}
\usepackage{amsthm}
\usepackage{mathtools}
\usepackage{amsfonts}
\usepackage{wasysym}
\usepackage{bbm}
\usepackage{thmtools} 
\usepackage{accents}
\usepackage{xargs}
\usepackage{enumitem}
\usepackage[hidelinks]{hyperref} 
\usepackage[capitalize, nameinlink]{cleveref}

\usepackage{graphicx}
\usepackage{xcolor}

\usepackage[margin=3.4cm]{geometry}

\definecolor{Bnavy}{RGB}{0, 66, 128}
\definecolor{Bdust}{RGB}{140,179,217}
\definecolor{Bsugarpaper}{RGB}{198, 217, 236}
\definecolor{Bgreen}{RGB}{142, 183, 114}
\definecolor{Blimegreen}{RGB}{202, 222, 189}
\definecolor{Bgreentheme}{RGB}{36, 87, 1}

\usepackage{makecell}

\usepackage{microtype}

\usepackage{tikz} 
\usepackage{tikz-cd}
\usetikzlibrary{ 
	calc,%
	arrows,%
	shapes,
	shapes.geometric,
    positioning,
    fit,
}

\tikzcdset{arrow style=tikz, diagrams={>=stealth}}

\usepackage[labelsep=quad,indention=10pt]{subfig}
\theoremstyle{plain}
\newtheorem{theorem}{Theorem}[section]

\newtheorem{lemma}[theorem]{Lemma}
\newtheorem{corollary}[theorem]{Corollary}
\newtheorem{proposition}[theorem]{Proposition}
\newtheorem{remark}[theorem]{Remark}
\newtheorem{notation}[theorem]{Notation}

\theoremstyle{definition}
\newtheorem{definition}[theorem]{Definition}
\newtheorem{example}[theorem]{Example}

\newenvironment{examplet}
    {\begin{example}
    }
    { 
    $\hfill \triangleleft$
    \end{example} 
    }

\usepackage[normalem]{ulem}

\newcommand{\Vect}{\mathrm{Vect}}

\newcommand{\A}{\mathcal{A}}

\newcommand{\C}{\mathcal{C}}

\newcommand{\Pos}{\mathcal{P}}

\newcommand{\fin}{\mathrm{fin}}

\newcommand{\utilde}[1]{\underaccent{\tilde}{#1}}
\newcommand{\define}[1]{{\bf \boldmath{#1}}}

\newcommand{\Field}{\mathbb{F}}

\newcommandx{\permodx}[2][1= \Pos, 2=\mathfrak X]{\mathrm{PerM}_{#2}(#1)}

\newcommand{\permodq}[1]{\mathrm{PerM}(#1)}

\newcommand{\eps}{\varepsilon}

\usepackage{extpfeil} 
\usepackage[all]{xy}

 \newlist{ClosProps}{enumerate}{4}
 \setlist[ClosProps]{label*=(\roman*)}
    \crefname{ClosPropsi}{Property}{Properties}
    \Crefname{ClosPropsi}{Property}{Properties}

\title{Notes on abelianity of \\ categories of finitely encodable persistence modules}

\author{Lukas Waas\footnote{Department of Mathematics, Heidelberg University, Germany, {\em email:} lwaas@mathi.uni-heidelberg.de}}

\date{}
\begin{document}
\maketitle
\begin{abstract}
    When working with (multi-parameter) persistence modules, one usually makes some type of tameness assumption in order to obtain better control over their algebraic behavior. One such notion is Ezra Millers notion of finite encodability, which roughly states that a persistence module can be obtained by pulling back a finite dimensional persistence module over a finite poset. From the perspective of homological algebra finitely encodable persistence have an inconvenient property: They do not form an abelian category. Here, we prove that if one restricts to such persistence modules which can be constructed in terms of topologically closed and sufficiently constructible (piecewise linear, semi-algebraic, etc.) upsets then abelianity can be restored.
\end{abstract}
This article was originally intended as a small set of notes hosted on my website, which I wrote in 2020 towards the beginning of my PhD (which was concerned with entirely different matters). After several people contacted me and said they would profit from a permanently accessible version on the arXiv, I decided to upload it in the current (slightly more polished) form.
\section{Introduction}
When working with (multiparameter) persistence modules, it is often necessary to make some type of tameness assumption in order to apply the methods of commutative and homological algebra, or representation theory (see for example \cite{chazal2014observable,Lesnick2015}).
In the one-parameter setting, for example, the classification theorem for persistence modules in terms of barcodes (\cite{Crawley-Boevey}) only holds under the assumption of pointwise finite dimensionality (see for ex \cite[Ex. 3.3.]{max_obs}). In \cite{Miller-hom}, the author introduced such a notion of tameness for the multiparameter setting, which he called finite encodability in a previous version of the paper\footnote{In the current version, the term \textit{tame} is used. We chose to stick with finite encodability, as it clearly distinguishes from other notion of tameness, such as finite presentability.}. Roughly speaking, a persistence module is finitely encodable if it can be obtained by pulling back a pointwise finite dimensional persistence module defined on a finite poset.
While this fairly general definition turns out to be quite powerful, it is somewhat deficient from a homological algebra point of view. Namely, the category of finitely encodable persistence modules over some fixed poset $\Pos$ is not a full abelian subcategory of the category of arbitrary persistence modules over $\Pos$ (\cite[Ex. 4.25]{Miller-hom}). For many applications, in particular to apply the language of amplitudes developed in \cite{WONG}, having the structure of an abelian category at hand is necessary. \\
Here, we show that under some slightly stronger constructability assumptions abelianity may be restored. In particular, we prove:
\begin{theorem}\label{thm:special_case_of_main_result}
    Let $\mathfrak X$ be the subset of the powerset of $\mathbb{R}^n$ generated under complement and union by the set of topologically closed upsets which are piecewise linear (semialgebraic, finitely subanalytic, or more generally obtained from some $o$-minimal structure as in \cite{van1998tame}).\\ 
    Let $\permodq{\mathbb R^n}$ be the category of all $n$-parameter persistence modules, with respect to some fixed field, and let $\permodx[\mathbb R^n]$ be the full subcategory given by such modules which are finitely encodable by an encoding map $e \colon \mathbb{R}^n \to \Pos$, which has fibers in $\mathfrak{X}$. 
    Then the inclusion \[
    \permodx[\mathbb R^n] \hookrightarrow \permodq{\mathbb R^n}\] makes $\permodx[\mathbb R^n]$ a full abelian subcategory of $\permodq{\mathbb R^n}$.
\end{theorem}
We obtain this result by showing that under certain connectedness assumptions on allowable encoding fibers, any two encodable persistence modules admit a common encoding which also encodes all morphisms between them (\cref{P:connective_alg_finest_enc}). This result is of interest on its own, as it frequently allows one to reduce a proof in the finitely encodable setting to the framework of a finite poset (see for example the results in \cite[Sec. 4]{WONG}). 
We note that one could have taken the alternative (albeit significantly less elementary) route, of obtaining a proof of the subanalytic case of \cref{thm:special_case_of_main_result} by passing to the world of sheaf theory, using results \cite{millerConstructible,KashiwaraShapPers,Berkouk_2021} (see \cref{rem:alternative_proof}). In our case, \cref{thm:special_case_of_main_result} follows from two theorems which may be formulated purely on the level of posets, not assuming any additional geometrical structure (\cref{thm:X_conn_abelian,P:connective_alg_finest_enc}). In particular, these apply to more general (for example discrete) scenarios and are of independent interest.
\section{Preliminaries and notation}
We begin by fixing some language and notation.
Aside from this, notation from \cite{Miller-hom} will be used freely.
\subsection{Notation}
\begin{itemize}
    \item By $\Vect$, we denote the category of vector spaces with respect to some fixed field $\Field$.
    The specific field itself will be immaterial to the discussion, and is hence omitted from the notation.
    \item  Given a small category $\C$ and another category $\A$, we denote by $\A^{\C}$ the category of functors with source $\C$ and target $\A$.
    We will mostly be interested in the particular case where $\C=\Pos$ is a poset (interpreted as a category with at most one morphism in every hom-set) and $\A= \Vect$.
    \item Elements of $\Vect^\Pos$ are called \define{persistence modules over $\Pos$}.
    \item By an \define{interval} in a poset $\Pos$ we mean a set $I \subset \Pos$ with the property that $q \leq p \leq q'$ and $q,q' \in I$ implies $p \in I$. Equivalently, an interval is a set given by the intersection between a downset and an upset in $\Pos$.
    \item Given two elements $q \leq q'$ of a poset $\Pos$, we denote by $[q, q']:= \{p \in \Pos \mid q \leq p \leq q' \}$ the interval of elements lying between $q$ and $q'$. We use analogous notation for half open intervals.
    \item We consider $\mathbb{R}^n$ as a poset, by equipping it with the product poset structure derived from the linear order on $\mathbb{R}$.
    \item Given an interval $I \subset \Pos$, we denote by $\Field[I] \in \Vect^{\Pos}$ the unique persistence module with $\Field$ at every point in $I$, $0$ everywhere else and the identity as structure morphism for pairs of points $p \leq p'$ in $I$.
    \item Given a functor $F: \C \to \C'$  we denote by $F^*: \A^{\C'} \to \A^{\C}$ the functor obtained by precomposition. 
    \item When $\A$ is a category that admits all small colimits, such as $\Vect$, then $F^*$ admits a left adjoint \[F_!\colon \A^{\C} \to \A^{\C'}\] given by left Kan-extension. Explicitly it is given by
    \[ F_!G(c') = \varinjlim_{F(c) \to c'}G(c) \]
    where the colimit is taken over the comma category $F/ c'$ and functoriality is induced in the obvious fashion.
\end{itemize}
\subsection{Finitely encodable persistence modules}
To be able to perform commutative and homological algebra on $\Vect^{\Pos}$ it is often necessary to assume some type of finiteness condition. In this paper, we are specifically concerned with following notion of finiteness:
\begin{definition}\cite{Miller-hom}
A \define{finite encoding} of a persistence module $M \in \Vect^\Pos$ consists of the following data:
\begin{enumerate}
    \item a finite poset $\Pos'$;
    \item a pointwise finite dimensional persistence module $M' \in \Vect^{\Pos'}$;
    \item a map of posets $e: \Pos  \to \Pos'$;
    \item an isomorphism $\phi: M \xrightarrow{\sim} e^*M'$.
\end{enumerate}
A persistence module $M \in \Vect^\Pos$ is called \define{finitely encodable}, if it admits a finite encoding\footnote{In \cite{WONG}, these are called finitely encoded. In \cite{Miller-hom}, in addition to the notion of encodability, the term tame is used.}. 
\end{definition}
For many purposes - not the least to turn finitely encodable modules into an abelian category - it is necessary to enforce some additional control over what kind of encodings one allows for. 
\begin{definition}
Let $\Pos$ be a poset. An \define{encoding structure} $\mathfrak X$ on $\Pos$ is a subset of the powerset of $\Pos$, such that
\begin{enumerate}
    \item $\mathfrak X$ is an algebra (i.e. closed under finite unions and complements, and contains $\Pos$);
    \item If $I \in \mathfrak X$ is an interval of $\Pos$, then there exist upsets $U,V \in \mathfrak X$, such that $I = U \cap V^c$;
    \item Every element of $\mathfrak X$ is a finite union of intervals which are themselves contained in $\mathfrak X$.
\end{enumerate}
\end{definition}
\begin{examplet}\label{ex:Enc_Structures}
Consider the case $\Pos= \mathbb R^n$ equipped with the encoding structure of staircases, generated by upsets of the form 
$ [p, \infty)$
for $p\in \mathbb (\mathbb R\cup \{ - \infty \})^{n}$. Intervals of this structure are given by finite unions of generalized cubes in $\mathbb R^n$. 
We denote this encoding structure by $\mathfrak C$ and call it the staircase structure. More generally, one can consider the encoding structure given by finite unions of intervals which are piecewise linear, semialgebraic or (finitely) subanalytic, or even more generally the intervals contained in a fixed o-minimal structure on $\mathbb R$ in the sense of \cite{van1998tame}.
\end{examplet}
Encoding structures allow one to restrict the considered category to persistence modules with tamer algebraic and topological behavior. 
\begin{definition}
Given an encoding structure $\mathfrak X$ on a poset $\Pos$, we say a persistence module $M \in \Vect^{\Pos}$ is \define{$\mathfrak X$-encodable} if it admits a finite encoding such that the encoding map $e: \Pos \to \Pos'$ has the property that every fiber of $e$ is in $\mathfrak X$. Such an encoding is also called \define{of class $\mathfrak X$}. We denote by $\permodx$ the full subcategory of $\Vect^{\Pos}$ given by $\mathfrak X$-encodable $\Pos$-persistence modules.
\end{definition}
\begin{remark}
    Note that a map $e \colon \Pos \to \Pos'$ has fibers in an encoding structure $\mathfrak X$, if and only if the inverse image of every upset in $\Pos'$ lies in $\mathfrak{X}$. This is part of the general philosophy of an encoding structure being determined by its upsets. In fact, the first and the third defining axiom guarantee that $\mathfrak{X}$ is an algebra generated by upsets. While one could have taken the path of defining encoding structure only in terms of their upsets, this makes expressing the second defining axiom and some of its consequences somewhat tedious. Furthermore, considering the way that encoding structures arise in practice (see \cref{rem:prop_of_enc_struct}), it seems to be more natural to think of them in terms of an algebra.
\end{remark}
\begin{remark}\label{rem:prop_of_enc_struct}
     The second axiom of an encoding structure is not used in the main result of these notes. It guarantees that an interval lies in $\mathfrak X$, if and only if it can be written as an intersection of an upset and a downset in $\mathfrak X$. This has the consequence that an interval module $\Field[I]$ is $\mathfrak X$-encodable, if and only if $I \in \mathfrak X$. The latter property is extensively used in the classification results of \cite[Sec. 4]{WONG}. 
\end{remark}
A direct consequence of the intersection stability of encoding structures is that $\permodx$ is always an additive subcategory of $\Vect^{\Pos}$. This follows from the following lemma also used in \cite{Miller-hom} and \cite{WONG}.
\begin{lemma}\label{L:common_encoding}
Let $M_1,\dots, M_n$ be finitely encodable persistence modules over $\Pos$. 
Then there exist an encoding map $e\colon \Pos \to \Pos'$ which is part of an encoding for $M_1, \dots, M_n$ simultaneously. 
Furthermore, if $M_1, \dots, M_n$ are $\mathfrak X$-encodable, for some encoding structure $\mathfrak X$ on $\Pos$, then $e$ can also be taken of class $\mathfrak X$.
\end{lemma}
\begin{proof}
We prove the case $n=2$, the general case is completely analogous.
Furthermore, we prove the second statement as it implies the earlier for the special case where $\mathfrak X$ is the encoding structure of all upsets in $\Pos$.
Choose $\mathfrak X$-encoding maps $e_1: \Pos \to \Pos_1, e_2: \Pos \to \Pos_2$ and encoding modules $M_1' \in \Vect^{\Pos_1}$, $M_2' \in \Vect^{\Pos_2}$. Now, set $e: \Pos \to \Pos_1 \times \Pos_2$ to the map induced by the universal property of the product and $M_1'':= \pi_{1}^*M'_1$,  $M_2'':= \pi_{2}^*M'_2$, where $\pi_i$ denotes the respective projection to $\Pos_i$. We obtain, 
    \[ M_i \cong e_i^*M'_i = (\pi_i \circ e)^*M''_i \cong e^*M''_i,\]
for $i=1,2.$
Since the fibers of $e$ are given by intersection of fibers of $e_1$ and $e_2$, $e$ again defines an $\mathfrak X$-encoding. 
\end{proof}
Using the addivity of the pullback functors, we immediately obtain from this:
\begin{corollary}
Let $\mathfrak X$ be an encoding stucture on a poset $\Pos$. Then $\permodx$ is a full additive subcategory of $\Vect^{\Pos}$.
\end{corollary}

\section{Connective encoding structures and abelianity of $\permodx$}
For many intents and purposes - for example to apply the framework developed in \cite{WONG} - it is desirable for $\permodx$ to be a full abelian  subcategory of $\Vect^{\Pos}$. However, the question of when this is the case is somewhat more subtle then the question of addivity, even when one restricts to nice enough encoding structures, such as the PL one.
Consider, for example, the following morphism of persistence modules described similarly in \cite[Ex. 4.25]{Miller-hom}.
\begin{examplet}\label{ex:bad_kernel}
Let $\nabla = \{ (x,y) \in \mathbb R^2 \mid x + y = 2\}$ be the antidiagonal in $\mathbb R^2$, shifted by $(1,1)$. Let $U=\{ (x,y) \in \mathbb R^2 \mid x + y \geq 2\}$ be the upset of points above (or equal to) $\nabla$. Next, consider the map of piecewise linearly encoded persistence modules $\Field[U]^2 \to \Field[\nabla]$ given by $0$ outside of $\nabla$, and by multiplication by the $1\times 2$ matrix $(y,-x)$ at $(x,y) \in \nabla$. Note that this indeed defines a morphism of persistence modules, as there are no nontrivial commutativity conditions to verify here. The kernel of this map is given by $\Field^2$ strictly above $\nabla$, and by the origin line going through $(x,y)$ at $(x,y) \in \nabla$. Note that this module can not be finitely encoded, since at every point $u$ strictly above $\nabla$ there are infinitely many different images coming from transition maps starting at $\nabla$ and ending in $u$.
\end{examplet}
Philosophically speaking, the problem with the last example is that the set $\nabla$, while perfectly tame as a topological space, is highly disconnected when considered as a poset (compare \cref{def:leq-con} to make this precise).
A certain amount of control over the path components (in a poset sense) of intervals generated by the encoding structure is required to obtain the necessary control over morphisms.
\begin{definition}\label{def:leq-con}
Let $\Pos$ be a poset. We say $\Pos$ is \define{$\leq$-connected} if it is connected as a category, that is, if for every $p,p' \in \Pos$ there exists a finite zigzag
\[ 
p \lessgtr p_1 \lessgtr  \dots \lessgtr  p_k \lessgtr  p'.
\]  
The maximal $\leq$-connected subsets of $\Pos$ are called its \define{$\leq$-connected components}. We say that an encoding structure $\mathfrak X$ on $\Pos$ is \define{connective}, if every interval $I$ in $\mathfrak X$ has only finitely many $\leq$-connected components and these are also elements of $\mathfrak X$.
\end{definition}
\cref{thm:X_conn_abelian} below states that for connective encoding structures the category $\permodx$ is indeed abelian. 
From the defining property of an interval, it is immediate that:
\begin{lemma}\label{lem:comp_of_int}
    The $\leq$-connected components of an interval $I \subset \Pos$ are themselves intervals of $\Pos$.
\end{lemma}
An easy elementary verification shows that the cubical structure $\mathfrak C$ on $\mathbb R^n$ is connective. More generally, one may show that under the additional assumption of being topologically closed (or open), all of the examples of \cref{ex:Enc_Structures} are connective, which is the content of \cref{cor:low_cl_conn}. We may now state one of the main results, which states that connectivity guarantees abelianity of $\permodx$.
\begin{theorem}\label{thm:X_conn_abelian}
Let $\mathfrak X$ be a connective algebra on the poset $\Pos$. 
Then the category $\permodx$ is a full abelian subcategory of $\Vect^{\Pos}$.
\end{theorem}
\begin{proof}
This is a special case of \cref{P:connective_alg_finest_enc} below.
Indeed, \cref{thm:X_conn_abelian} follows from the following easily verified fact of homological algebra:  Let $\mathcal B$ be a full additive subcategory of an abelian category $\mathcal A$. Then, $\mathcal{B}$ is an abelian subcategory of $\mathcal{A}$, if and only if every morphism in $\mathcal B$ is contained in some full abelian subcategory $\mathcal C$ of $\mathcal{A}$, such that $\mathcal C \subset \mathcal B$.
\end{proof}
The following proposition guarantees that when $\mathfrak X$ is connective, then essentially all finite computations in $\permodx$ may instead be performed in $\Vect^{\Pos'}$ over some finite poset $\Pos'$.
\begin{proposition}
\label{P:connective_alg_finest_enc}
Let $\mathfrak X$ be a connective encoding structure on a poset $\Pos$, and $M_0, \dots, M_n$ a finite set of $\mathfrak X$-encodable persistence modules.
Then there exists a common $\mathfrak X$-encoding map $e \colon \Pos \to \Pos'$ of $M_0, \cdots, M_n$ such that $e^* \colon \Vect^{\Pos'} \to \Vect^{\Pos}$ is fully faithful. \\
In particular, since $e^*$ is exact, the essential image $\mathcal I = e^{*}(\Vect^{\Pos'}_{\fin})$ is a full abelian subcategory of $\Vect^{\Pos}$ such that $M_0, \dots, M_n \in \mathcal I \subset \permodx$.
\end{proposition}
Before we provide a proof, note that \cref{P:connective_alg_finest_enc} does in particular provide a recipe on how to compute finite limits and colimits in $\permodx$.
To prove \cref{P:connective_alg_finest_enc}, we first need to investigate when the functor $e^*$ is fully faithful.

\begin{proposition}\label{prop:estar_ff}
Let $e: \Pos \to \Pos'$ be a map of posets such that:
\begin{enumerate}
    \item The relation $\leq$ on $\Pos'$ is generated under transitivity by the images of the relations in $\Pos$ under $e$;
     \item All fibers of $e$ are nonempty and $\leq$-connected.
\end{enumerate}
Then the functor $e^* \colon \mathcal \Vect^{\Pos'} \to \Vect^{\Pos}$ is fully faithful.
\end{proposition}

\begin{proof}
Recall the basic fact from category theory that, given an adjunction of functors $L \dashv R$, the right adjoint is fully faithful, if and only if the counit of adjunction $\varepsilon \colon LR \to 1$ is an isomorphism. In the case of the adjunction $e_! \dashv e^*$, the counit $\varepsilon_M$  for $M \in \Vect^{\Pos'}$ at $q' \in \Pos'$ is given by the canonical morphism
\[\varepsilon:\varinjlim_{q \in e^{-1}(-\infty, q']}  M_{e(q)} \to M_{q'} \, .
\]  
In particular, it suffices to show that, under the assumptions, $\varepsilon$ is an isomorphism. 
Consider the subdiagram of the diagram indexed over $e^{-1}(-\infty, q']$ that is given by restricting to $e^{-1}(q')$.
This is now a constant $M_{q'}$-valued diagram over a nonempty, connected category. In particular, the natural map 
\[ 
\varinjlim_{q \in e^{-1}\{q'\}}  M_{e(q)} \to M_{q'} 
\] 
is an isomorphism. The latter map fits into the commutative triangle
\[
\begin{tikzcd}
    \varinjlim_{q \in e^{-1}\{q'\}} \arrow[rd] M_{e(q)} \arrow[r] & M_{q'} \\
    &\varinjlim_{q \in e^{-1}(-\infty,q']}  M_{e(q)} \arrow[u, "\varepsilon_{q'}"]
\end{tikzcd}.
\]
In particular, it follows that $\varepsilon_{q'}$ is surjective.
Hence, if we can show that 
\begin{equation}\label{E:sur_col}
\varinjlim_{q \in e^{-1}\{q'\}}  M_{e(q)} \to \varinjlim_{p \in e^{-1}(-\infty,q']}  M_{e(p)}
\end{equation} 
is surjective, then $\varepsilon_{q'}$ is injective, and thus an isomorphism. 
Consider any generator of the colimit on the right, given by the equivalence class of some $v \in M_{e(p)}$ with $e(p) \leq q'$. 
By the first assumption, the relation $e(p) \leq q'$ is obtained from some sequence 
\[ 
p_{0} \leq \tilde p_{1}; p_{1} \leq \tilde p_{2}; \dots; p_{k-1} \leq \tilde p_{k} 
\] 
such that $e(p_i) = e(\tilde p_i)$, $e(p) = e( p_0)$ and $e(\tilde p_k) = q'$. 
Assume we have shown that $v$ has the same equivalence class as some element $\tilde v_l \in M_{\tilde p_l}$, for some $l \leq k$. 
By the connecteness assumption on the fibers, there exists a zigzag from $\tilde p_l$ to $p_l$. 
As the diagram is given by isomorphisms (the identity) on each fiber, $\tilde v_l$ is identified with some $v_l \in M_{p_l}$ along this zigzag. 
The latter is identified with some $\tilde v_{l+1} \in M_{\tilde p_{l+1}}$ by the relation $p_l \leq \tilde p_{l+1}$. 
By induction, $v$ is ultimately identified with some element in $M_{e(q)}$, for $q$ in the fiber of $q'$.
In particular, its equivalence class lies in the image of the colimit of the diagram restricted to $e^{-1}(q')$, showing the surjectivity of the map in (\ref{E:sur_col}) and thus the required injectivity of $\varepsilon_{q'}$.
\end{proof}

Next, we show that in the case of a connective encoding structure $\mathfrak X$, any $\mathfrak X$-encoding can be replaced by an encoding that fulfills the requirements of \cref{prop:estar_ff}.

\begin{lemma}\label{lem:improve_phi}
Let $e\colon \Pos \to \Pos'$ be a map of posets such that each of its fibers has finitely many $\leq$-components. 
Then $e$ admits a factorization 
\[ 
\begin{tikzcd}
    \Pos \arrow[rr, "e"] \arrow[rd, "\hat e"']&& \Pos' \\
    & \hat \Pos \arrow[ru, dashed]& 
\end{tikzcd}
\]
where $\hat e$ is a map fulfilling the requirements of \cref{prop:estar_ff}. 
Further, $\hat e$ can be taken so that its fibers are precisely the $\leq$-connected components of the fibers of $e$. In particular, if $e$ is a $\mathfrak X$-encoding map where $\mathfrak X$ is a connective encoding structure, then $\hat e$ is also an $\mathfrak X$-encoding map.
\end{lemma}

\begin{proof}
We take $\hat \Pos$ to be the set given by the $\leq$-components of the fibers of $e$. 
We take the partial order on $\hat \Pos$ to be the one generated by the following relations: 
For $I \subset \varphi^{-1}(p')$ and $ J \subset \varphi^{-1}(q')$, we set $I \lesssim J$ if and only if there exists $p \in I$ and $q \in J$ fulfilling $i' \leq j'$.
To show that this indeed induces a partial ordering on $\hat \Pos$, i.e. that anti-symmetry is fulfilled, we need to show that $\lesssim$ admits no cycles. 
So, suppose we are given a sequence in $\Pos$
\[
p_{0} \leq \tilde p_{1}; p_{1} \leq \tilde p_{2}; \dots; p_{k} \leq \tilde p_{0} 
\]
such that, for $0 \leq l \leq k$, $\tilde p_{l}$ and $p_{l}$ lie in the path component of the fiber of $e(p_l)=e(\tilde p_l)$, respectively. 
By applying $e$ and using anti-symmetry on $\Pos'$, we obtain that all of the $p_l$ and $\tilde p_l$ lie in the same fiber. 
As $p_{l}$ and $\tilde p_l$ lie in the same path component, we can complete the sequence by filling in zigzags in the respective path components in between $p_l$ and $\tilde p_l$. 
This gives a zigzag between $p_{0}$ and $\tilde p_0$ lying entirely in the same fiber, showing all of the $p_l$ and $\tilde p_l$ indeed belong to the same component.
Then, by construction, the map $\hat e\colon\Pos \to \hat \Pos$ given by sending each element to the respective component of the fiber it is contained in, defines a map of partially ordered sets, fulfilling the conditions of \cref{prop:estar_ff}, with the dashed factorization map just being given by sending each component $ I \subset e^{-1}(q)$ to $q$. 
\end{proof}
We now have all the tools necessary to the proof of \cref{P:connective_alg_finest_enc}:
\begin{proof}[Proof of \cref{P:connective_alg_finest_enc}]
    By \cref{L:common_encoding}, we may choose a common encoding map $e' \colon \Pos \to \Pos'$, for $M_0, \dots, M_n$ of class $\mathfrak X$. Now, apply \cref{lem:improve_phi} to $e'$, to obtain an $\mathfrak X$-encoding map $\hat e$, fulfilling the requirements of \cref{prop:estar_ff}. Since, $\hat e$ factors through $e'$, it still encodes $M_0, \dots, M_n$. 
\end{proof}

\subsection{Examples of connective encoding structures}
In this subsection, we show that connective encoding structures on $\mathbb R^n$ are ubiquitous, and arise naturally by also taking the topology of $\mathbb R^n$ into account. More precisely we prove \cref{cor:low_cl_conn} which states that closed PL, semialgebraic, or more generally upsets in some o-minimal structure on $\mathbb R^n$ generate a connective encoding structure.
\cref{thm:special_case_of_main_result} then follows from \cref{cor:low_cl_conn} together with \cref{thm:X_conn_abelian}. Let us begin by taking a look at how connective components in the $\leq$ sense interact with connective components in the topological sense.
\begin{definition}
    If $\Pos$ is a poset equipped with the structure of a topological space, we say $\Pos$ is \define{locally $\leq$-connected} if it admits a neighborhood basis by $\leq$-connected sets.
\end{definition}
We will make use of the following easily proven lemma.
\begin{lemma}\label{lem:leq_con_is_con}
Let $\Pos$ be a poset equipped with the structure of a topological space. If $\Pos$ is locally $\leq$-connected, then the topological connected components of $\Pos$ refine the $\leq$-connected components of $\Pos$. Conversely, if $\Pos$ is such that every interval $[p,q]$ is topologically connected, then the $\leq$-components refine the topological components of $\Pos$.
\end{lemma}
In particular, for the case of the topological poset $\mathbb R^n$ we obtain:
\begin{corollary}\label{prop:leq_and_top_comp}
Let $I = U \cap D$ be the intersection of an upset and a downset in $\mathbb R^n$, one of which is open. Then the topological connected components of $I$ and the $\leq$-connected components of $I$ agree. 
\end{corollary}
\begin{proof}
We prove the case when $U$ is open. For $v \in  I$, consider a vector $\varepsilon > 0$ such that the $\varepsilon$-cube around $v$ in the maximum norm, $C_{\eps}$, lies in $U$. Then every $u \in C_{\eps} \cap D \subset I$ lies above $v - \varepsilon(1,\dots,1) \in I$. In particular, $C_{\eps} \cap D$ is a $\leq$-connected neighborhood of $v$ in $I$. These sets form a neighborhood basis of $I$, showing that $I$ is locally $\leq$-connected. By \cref{lem:leq_con_is_con}, the statement follows.
\end{proof}

\begin{notation}\label{not:closed_encoding_struct}
 Given an encoding structure $\mathfrak X$ on $\mathbb R^n$, we denote by $\underline{\mathfrak X}$ the subalgebra generated by such upsets in $\mathfrak X$ which are topologically closed.
\end{notation}
The goal is to show that for most scenarios of interest $\underline{\mathfrak{X}}$ is a connective encoding structure. To do so, let us introduce some more notation.
\begin{notation}
     Let $S \subset \mathbb R^n$. We denote by $\underline{S}$ the set of limit points of sequences $(x_k)_{k \in \mathbb N}$, $x_k \to x$, with $x_k \geq x$ and $x_k \in S$. Furthermore, we denote $\widetilde{S}:=(\underline{ (S^c)})^c$. 
\end{notation}
Next, let us list some of the elementary properties of the operation $\underline{(-)}$, which we are going to use to investigate when the closed upsets of an encoding structure again generate an encoding structure.
\begin{lemma}\label{lem:elem_facts_upsetsclosures}
The following properties of the operations $\underline{(-)}$ and $\widetilde{(-)}$ hold:
        \begin{ClosProps}
            \item \label{lem:elem_facts_upsetclosure_1} If $U$ is an upset of $\mathbb R^n$ , then $\underline U =\overline{U}$, the topological closure of $U$, and $\overline{U}$ is again an upset.
            \item \label{lem:elem_facts_upsetclosure_2}If $D$ is a downsets of $\mathbb{R}^n$, then $\widetilde{D} = \mathring{D}$, the topological interior of $D$, and $\mathring{D}$ is again a downset.
            \item \label{lem:elem_facts_upsetclosure_3}If $U,V \subset \mathbb R^n$ are upsets and $V$ is closed, then $\underline{U \cap V^c} = \underline{U} \cap V^c$.
            \item \label{lem:elem_facts_upsetclosure_4}If $U,V \subset \mathbb R^n$ are upsets, then
            $\widetilde{({U \cap V^c})} = U \cap \widetilde{(V^c)}$.
        \end{ClosProps}
\end{lemma}
\begin{proof}
    To see that \cref{lem:elem_facts_upsetclosure_1} holds,
    let $(x_k)_{k \in \mathbb N}$ be a sequence in $U$ which converges to $x \in \mathbb R^n$. 
    By replacing $x_k$ with $\sup(x_k,x)$, we may without loss of generality assume that $x_k \geq x$, for all $k \in \mathbb N$. This shows that $\overline{U}$ may indeed be described as in the statement of the lemma. 
    Now, let $x \in \overline{U}$ and $y \in \mathbb R^n$ such that $x \leq y$. Then, for any sequence $(x_k)_{k \in \mathbb N}$ in $U$ converging to $x$ from above, the sequence $(x_k + (y-x))_{k \in \mathbb N}$ also lies in $U$ and converges to $y$, which shows that $\overline{U}$ is indeed an upset. \\
    The second property follows from the first by taking complements.
    For \cref{lem:elem_facts_upsetclosure_3}, note first that as $V^c$ is a downset, we have $\underline {(V^c)} = V^c$. Hence, it follows that 
    \[
    \underline{U \cap V^c} \subset \underline{U} \cap \underline{V^c} = \underline{U} \cap V^c.
    \]
    Since $V^c$ is open, any sequence in $U$ converging to $x \in V^c$ ultimately lies in $V^c$, which shows
    \[ \underline{U} \cap V^c \subset \underline{U \cap V^c} .\]
   Finally, to prove \cref{lem:elem_facts_upsetclosure_4}, we may equivalently show that 
    \[
    \underline {U^c \cup V} = U^c \cup \underline{V}.
    \]
    This is immediate, from the fact $\underline{(-)}$ commutes with unions together with $U^c$ being a downset.
\end{proof}
As an immediate consequence, we obtain:
\begin{lemma}\label{lem:invariance_under_clos_below}
    Let $S=(U_1 \cap V^c_1) \cup \cdots \cup (U_n \cap V^c_n)$ with $U_i,V_i \subset \mathbb R^n$ upsets, which are topologically closed. Then the equalities
    \[
    \underline{S}= S = \widetilde{S} 
    \]
    hold.
\end{lemma}
\begin{proof}
    The first equality is immediate from \cref{lem:elem_facts_upsetclosure_1,lem:elem_facts_upsetclosure_3} of \cref{lem:elem_facts_upsetsclosures}, together with commutativity with unions. For the second equality, the nontrivial part is showing that
    $S \subset \widetilde S$. Since $\widetilde{(-)}$ preserves inclusions, it suffices to show
    \[U_i \cap V_i^c \subset \widetilde{(U_i \cap V_i^c)}\]
    and we may further reduce to the case where $n=1$. The latter is immediate from properties \cref{lem:elem_facts_upsetclosure_2,lem:elem_facts_upsetclosure_4} of \cref{lem:elem_facts_upsetsclosures}.
\end{proof}
Furthermore, we are going to make use of the following property of $\underline{(-)}$ and $\widetilde{(-)}$.
\begin{lemma}\label{lem:lower_clos_of_disj}
    Suppose that $S \subset \mathbb R^n$ is such that $\underline S = S = \widetilde S$ and let 
    \[
    S = S_1 \sqcup \cdots \sqcup S_n
    \]
    be a decomposition into sets which is $\leq$-disconnected, i.e. there are no relations $x_i \leq x_j$, for $x_i \in S_j, x_j \in S_j $ and $i \neq j$. Then, for each $j \in \{1,\dots,n\}$, the equalities
    \[
    \underline{S_j} = S_j = \widetilde{S_j}
    \]
    hold.
\end{lemma}
\begin{proof}
     Suppose that $i$ is such that there exists an $x \in \underline{S_i}$ with $x \notin S_i$. Since $\underline {S_i} \subset  \underline{S} = S$, it follows that $x \in S_j$ for some $j \neq i$. Consquently, there exists a sequence $(x_k)_{k \in \mathbb N}$ in $\underline{S_i}$, with $x_k \geq x$. This stands in contradiction with the incomparability assumption between $S_i$ and $S_j$. Similarly, assume that $x \in S_i$, but $x \notin \widetilde{S_j}$. Then, by definition there exists a sequence $(x_k)_{k \in \mathbb N}$ with $x_k \in (S_i)^c$ and $x_k \geq x$, converging to $x$. However, since $\widetilde{S} = S$ and $x \in S$, $x_k$ has to be contained in $S$ for $k$ sufficiently large. In particular, this implies that at least some $S_j$ with $i \neq j$ contains an $x_k$, for some $k$ sufficiently large. Again, this stands in in contradiction to the $\leq$-disjointness assumption.
\end{proof} 
We may then show the following proposition, which guarantees that the encoding structures we are mainly interested in, such as the PL, semialgebraic and more general o-minimal ones behave well with restricting to closed upsets.
\begin{proposition}\label{prop:closeds_form_enc}
    In the situation of \cref{not:closed_encoding_struct}, suppose that $\mathfrak X$ is closed under taking topological closure of upsets. Then $\underline{\mathfrak X}$ is again an encoding structure. The intervals in $\underline{\mathfrak X}$ are precisely the intervals $I \in \mathfrak X$, for which $\underline I = I = \widetilde I$. In particular, the upsets in $\underline{\mathfrak X}$ are precisely the closed upsets in $\mathfrak X$.
\end{proposition}
\begin{proof}
    The only involved part of the proof is showing that any interval
    \[
    I = (U_1 \cap V^c_1) \cup \cdots \cup (U_n \cap V^c_n),
    \]
    with $U_i,V_i \in \mathfrak X$ closed upsets of $\mathbb R^n$, may again be written in the form $U' \cap V'^c$, with $U',V' \in \mathfrak X$ closed upsets.
    Since $\mathfrak X$ is an encoding structure, we may write
    \[
    I= U \cap V^c,
    \]
    for upsets $U,V \in \mathfrak X$. We claim that 
    \[
    I = \overline{U} \cap \overline{V}^c,
    \]
    holds, which finishes the first part of the proof, by the assumption that $\mathfrak X$ is closed under taking topological closures of upsets. Now, to see that $I = \overline{U} \cap \overline{V}^c$, note first that there are  inclusions
    \[
    \widetilde{I} \subset U \cap \widetilde{(V^c)} \subset I.
    \]
    By \cref{lem:invariance_under_clos_below} we have 
    $
    \widetilde{I} = I .
    $
    and hence 
    \[
    I = U \cap \widetilde{(V^c)}.
    \]
    Consequently, we may assume without loss of generality that $V = \underline{V}$, i.e. by \cref{lem:elem_facts_upsetclosure_1} of \cref{lem:elem_facts_upsetsclosures}, that $V$ is closed. Thus, we may now apply \cref{lem:elem_facts_upsetclosure_3} of \cref{lem:elem_facts_upsetsclosures} together with \cref{lem:invariance_under_clos_below}, to obtain
    \[
    I = \underline{I} = \underline{U} \cap V^c,
    \]
    as was to be shown. Note, that the only two properties used to write $I$ in the form $I = \overline{U} \cap \overline{V}^c$ were that $\underline I = I = \widetilde{I}$. This yields the characterization of intervals in $\underline {\mathfrak X}$ in the statement of the proposition.
\end{proof}
We may now combine \cref{prop:closeds_form_enc} with  \cref{prop:leq_and_top_comp}  and \cref{lem:lower_clos_of_disj} to show the following result:
\begin{corollary}\label{cor:low_cl_conn}
If $\mathfrak X$ is any encoding structure on $\mathbb R^n$, which is closed under taking closures of upsets, and is such that any of its intervals has only finitely many topological components and these are again in $\mathfrak X$, then $\underline{\mathfrak X}$ is a connective encoding structure. 
\end{corollary}
\begin{proof}
    By \cref{prop:closeds_form_enc}, $\underline{\mathfrak X}$ does indeed form an encoding structure.
    Now, if $I \in \underline{\mathfrak X}$ is an interval, then by assumption we may write $I$ as a topologically disjoint union $I = I_1 \sqcup \dots \sqcup I_n$ with $I_j$ elements of $\mathfrak X$, which are topologically connected.
    Since, by \cref{prop:closeds_form_enc}, $I$ is the intersection of a closed upset with an open downset, \cref{prop:leq_and_top_comp} implies that $I$ is locally $\leq$-connected and hence the $I_j$ are also the $\leq$-connected components of $I$. In particular, \cref{lem:comp_of_int}, the $I_j$ are again intervals.
    It remains to show that they are indeed elements of $\underline {\mathfrak X}$. This now follows by the characterization of intervals of \cref{prop:closeds_form_enc} together with \cref{lem:lower_clos_of_disj}.
\end{proof} 
\begin{remark}
     In particular, the assumptions of \cref{cor:low_cl_conn} are fulfilled, when $\mathfrak X$ is given by the set of finite unions of PL or semialgebraic intervals (or alternatively any encoding structure derived from an o-minimal structure in the sense of \cite{van1998tame}). Indeed, in these scenarios the number of topological components of each interval is finite and they are again of the respective class (see \cite[Prop. 2.18]{van1998tame}). Furthermore, since any interval $[a,b] \subset \mathbb R^n$ is connected, it follows that the topological components of any interval are themselves intervals.
\end{remark}
Now, \cref{thm:special_case_of_main_result} is simply the combination of \cref{cor:low_cl_conn} together with \cref{thm:X_conn_abelian}.
\begin{remark}\label{rem:alternative_proof}
    Note first that instead of working with closed sets, in the definition of $\underline{\mathfrak X}$, one may just as well work with open ones, and obtains a corresponding version of \cref{cor:low_cl_conn}.\\ One may use this, to rephrase the results of this paper in terms of alternative descriptions of persistence modules, and their categories of observables.
    For the remainder of this remark, fix some encoding structure $\mathfrak X$ on $\mathbb R^n$ which is closed under taking interiors of upsets, and denote by $\utilde{\mathfrak X}$ the encoding structure generated by the open upsets of $\mathfrak X$.
    While there are some details to be verified, conjecturally the following relationship between the observable perspective introduced in \cite{Berkouk_2021} and sheaf theoretic models for persistence modules (see \cite{KashiwaraShapPers}) should hold. The subanalytic case is discussed in \cite{millerConstructible}.
    \begin{enumerate}
        \item Only allowing for open subsets in $\mathfrak X$ essentially amounts to passing to a specific subcategory of $\gamma$-sheaves, as defined in \cite{KashiwaraShapPers}. Namely, to those $\gamma$-sheaves which are constructible with respect to a finite stratification of $\mathbb R^n$ by elements of $\mathfrak X$.
        \item Consequently, under the equivalence between $\gamma$-sheaves and the observable category of \cite{Berkouk_2021}, $\permodx[\mathbb R^n][\utilde{\mathfrak X}]$ should be equivalent to the full subcategory of the observable category of persistence modules which are $\mathfrak X$-encodable, i.e. isomorphic to an object in $\permodx[\mathbb R^n]$ in the observable category.
        \item Finally, in the language of sheaves the category $\permodx[\mathbb R^n][\utilde{\mathfrak X}]$ should correspond to the category of sheaves which are constructible with respect to a finite stratification by elements of $\mathfrak X$, and have microsupport in the negative polar cone $\gamma^{o,a}$, where $\gamma$ denotes the positive cone $\mathbb R_{\geq 0}^n \subset \mathbb R^n$ (following the notation of \cite{KashiwaraShapPers}).
    \end{enumerate}
    Hence, an alternative proof of \cref{thm:special_case_of_main_result} should follow by verifying the abelianity of the final category in this list.
\end{remark}

\section*{Acknowledgments}
I would like to thank Ezra Miller and Barbara Giunti for helpful discussions and feedback, as well as the \textit{Landesgraduiertenförderung Baden-Württemberg} for their financial support.
\newpage
\small
\bibliographystyle{alpha}
\bibliography{references}
\end{document}